\documentclass{amsart}
\usepackage[pagewise]{lineno}
\usepackage[utf8]{inputenc}
\usepackage[T1]{fontenc}
\usepackage[margin=01.45 in]{geometry}
\usepackage{dirtytalk}
\usepackage[dvipsnames]{xcolor}
\usepackage{float}

\usepackage{amsmath}
\usepackage{amssymb}
\usepackage{amsthm}
\usepackage{tikz-cd}
\usepackage{tikz}
\usetikzlibrary{knots, snakes}
\usetikzlibrary{arrows.meta}
\usepgflibrary[arrows.meta]
\usetikzlibrary{angles, quotes, decorations.markings, decorations.pathmorphing, decorations.pathreplacing, decorations.shapes}
\usepackage{graphicx}
\usepackage{mathtools}
\usepackage{extarrows}
\usepackage{booktabs}
\usepackage{pythonhighlight}
\usepackage{hyperref}
\usepackage[all]{xy}
\usepackage{accents}
\usepackage[backend=biber,style=numeric,sorting=nyt, maxbibnames=99]{biblatex}
\usepackage{amsmath,amsthm,amssymb,latexsym,amscd}
\usepackage{mathrsfs}

\usepackage{hyperref}
\hypersetup{colorlinks=true,linkcolor=blue,anchorcolor=blue,citecolor=blue}
\numberwithin{equation}{section}
\theoremstyle{plain}
 
\newtheorem{lemma}[equation]{Lemma} 
\newtheorem*{lemma*}{Lemma}

\newtheorem{proposition}[equation]{Proposition}
\newtheorem*{proposition*}{Proposition}

\newtheorem{theorem}[equation]{Theorem}
\newtheorem*{theorem*}{Theorem}

\newtheorem{corollary}[equation]{Corollary}
\newtheorem*{corollary*}{Corollary}

\newtheorem*{conjecture*}{Conjecture}

\definecolor{purp}{RGB}{148,30,238}
\definecolor{pinky}{RGB}{255, 174, 252}

\definecolor{purp1}{RGB}{120,30,238}
\definecolor{pinky1}{RGB}{243, 162, 252}
\definecolor{pinky2}{RGB}{252, 165, 249}
\definecolor{pinky3}{RGB}{247, 174, 255}

\definecolor{purp2}{RGB}{189, 144, 249}
\definecolor{purp3}{RGB}{203, 172, 243}

\theoremstyle{remark}

\newtheorem{remark}[equation]{Remark}

\theoremstyle{definition}

\newtheorem{definition}[equation]{Definition}

\newenvironment{enumalph}
{\begin{enumerate}}
	{\end{enumerate}}

\DeclareMathOperator{\coeff}{coeff}
\DeclareMathOperator{\lead}{lead}
\DeclareMathOperator{\term}{term}

\title{A Condition on the Jones Polynomial for a Family of Positive Links}
\author{Lizzie Buchanan}
\address{Department of Mathematics, The University of Iowa, Iowa City, IA 52242}
\email{elizabeth-buchanan@uiowa.edu}
\urladdr{\url{https://sites.google.com/view/lizziebuchanan/}}

\begin{document}
	
	\maketitle
	
	\begin{abstract}
		We provide a bound on the maximum degree of the Jones polynomial of any positive link with second Jones coefficient equal to $\pm 1$. This builds on the result of our previous work, in which we found such a bound for positive fibered links.
	\end{abstract}

	\section{Introduction}
	
	
	A positive knot diagram hints at its own existence by forcing certain properties on its knot polynomials and other invariants. For example, the Conway polynomial of a positive knot is positive (\cite{Cromwell}) and the signature of a positive knot is negative (\cite{przytycki2009positive}). So, any of these properties can be used as a positivity obstruction: if a given knot cannot satisfy these conditions, the knot cannot be positive. 
	
	Most of our positivity obstructions fail, however, to distinguish almost-positive knots from positive knots, because many results about positive knots have been shown to be true for almost-positive knots as well. For example, the Conway polynomial of an almost-positive knot is positive (\cite{Cromwell}) and the signature of an almost-positive knot is negative (\cite{przytycki2009almostpositive}). It is also now known that both positive and almost-positive knots are strongly quasipositive (\cite{Feller_2022}). We refer the reader to \cite{Feller_2022} and \cite{StI} for more information on the many overlapping properties of positive and almost-positive knots.

	In \cite{Buchanan_2022}, we found a condition on the Jones polynomial of a fibered positive knot that can, in some cases, show that a given knot is almost-positive rather than positive. 
	
	\begin{theorem} \cite{Buchanan_2022}\label{main result for coeff 0}
		The Jones polynomial of a fibered positive $n$-component link $L$ satisfies
		$$\max \deg V_L \leq 4 \min \deg V_L + \frac{n-1}{2}.$$
	\end{theorem}
	
	Among positive links, a fibered positive link is distinguished by having $0$ as the second coefficient of its Jones polynomial, meaning that the $t^{\min\deg V_L + 1}$ term vanishes  \cite{futer2012guts} \cite{Stoimenow}. This paper focuses on the next class of positive links: those characterized by having $\pm 1$ as the second coefficient of their Jones polynomial. We prove the following result. \\
	
	\noindent \textbf{Theorem \ref{main result coeff 1}.} \textit{Let $L$ be a positive link with $n$ link components, Jones polynomial $V_L$, and Conway polynomial $\nabla_L$. If the term $t^{\min\deg V_L + 1}$ appears in $V_L$ with coefficient $\pm 1$, then  $$\max \deg V_L \leq 4 \min \deg V_L + \frac{n-1}{2} + 2\lead \coeff \nabla_L -2, $$
		where $\lead \coeff \nabla_L$ is the coefficient of the highest degree term in the Conway polynomial.}\\

	This result can be used as a positivity obstruction. For example, in Figure \ref{12n149!} we have an almost-positive diagram of knot $12_n149!$, the mirror of the knot listed on the KnotInfo website as $12_n149$. Knot $12_n149!$ has Jones polynomial $t^2 - t^3 + 2t^4 - 2t^5 + 2t^6 - t^7 + t^9 - t^{10} + t^{11} - t^{12}$ and Conway polynomial $1+7z^2 + 2z^4$. This information is listed on KnotInfo and was verified for this particular diagram using Regina \cite{knotinfo} \cite{Regina}. We observe that $12_n149!$ does have $-1$ as the second coefficient of its Jones polynomial and yet its polynomials do not satisfy the inequality in Theorem \ref{main result coeff 1}. Therefore we conclude that $12_n149!$ is not a positive knot. 
	
	\begin{figure}
		\center 	{\includegraphics[width=4.5cm]{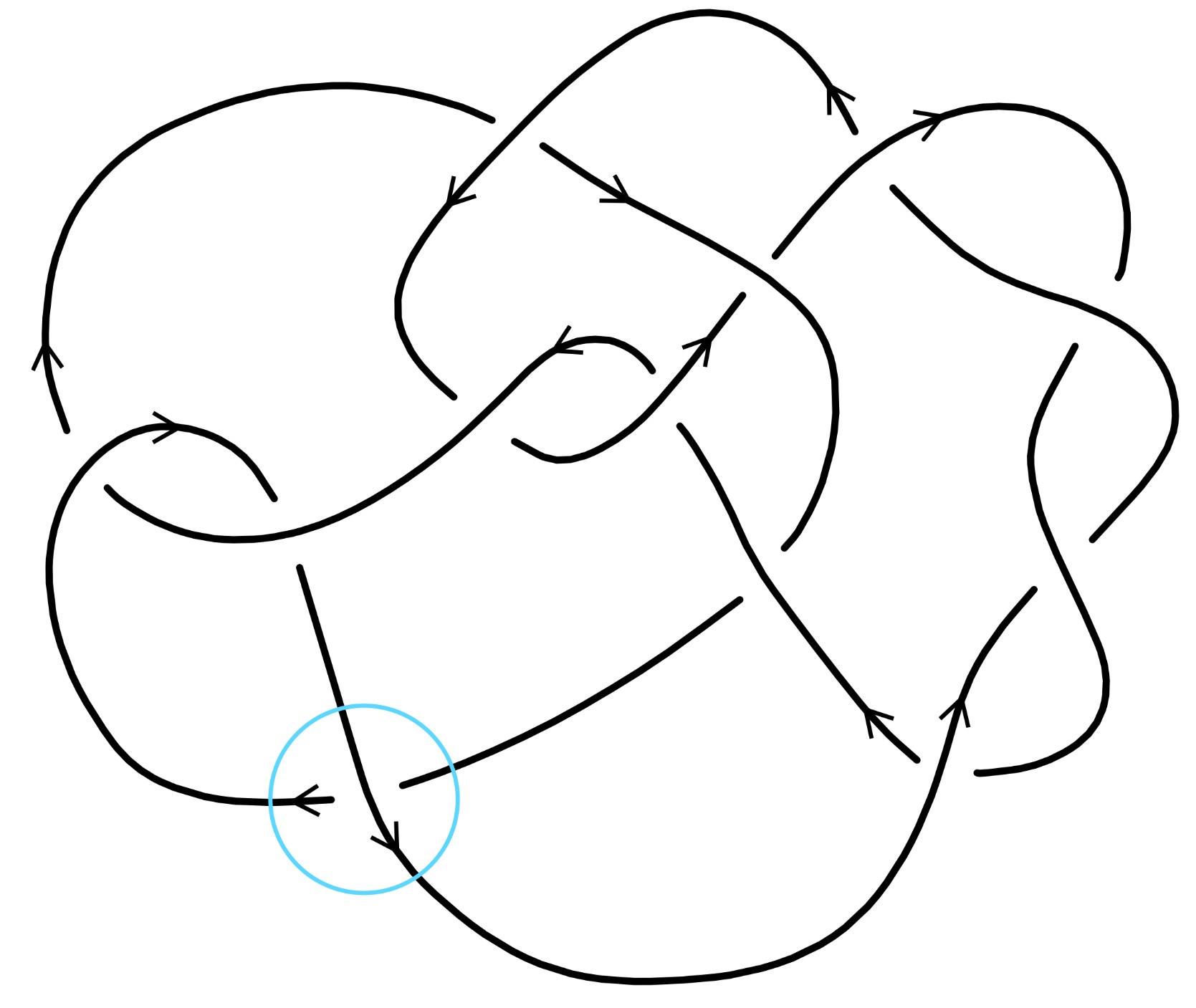}}
		\caption{ An almost-positive diagram of almost-positive knot $12_n149!$. Negative crossing circled. DT Code: [-4, -20, 16, 14, 22, 18, 8, 6, 24, 10, -2, 12]} \label{12n149!}
	\end{figure}
	
	The sequel to this paper will develop a similar positivity obstruction for links with second Jones coefficient equal to $\pm 2$. An earlier version of that paper and this one appears on the arXiv as \cite{Buchanan2023}. There, we also provide an infinite family of examples of almost-positive knots for which Theorem \ref{main result coeff 1} obstructs positivity. 
	
	\subsection*{Acknowledgements}
	The author would like to thank her PhD advisor at the Dartmouth College Department of Mathematics, Vladimir Chernov. Additional thanks to Keiko Kawamuro, Patricia Cahn, and Marithania Silvero for helpful conversations, and to an anonymous reviewer for comments and suggestions. The author is partially supported by NSF Grant DMS-2038103 with Keiko Kawamuro at the Univeristy of Iowa. 
	
	\section{Background and Definitions}

	\begin{definition}
		The \underline{\textbf{second coefficient of the Jones polynomial}} is the coefficient of the $t^{\min\deg V_L + 1}$ term in the Jones polynomial. For example, the Jones polynomial of the trefoil is $V_{3_1}(t) = t + t^3 - t^4$, so we have $\min \deg V_{3_1} = 1$, $\max \deg V_{3_1} = 4$, and the \textbf{second Jones coefficient} is $0$. 
	\end{definition}

	\begin{definition}
		A (oriented) link diagram is called \underline{\textbf{positive}} if every crossing in that diagram is positive, as in Figure \ref{positive and negative crossing}. A link is \underline{\textbf{positive}} if it has a positive diagram.  
	\end{definition}
	
	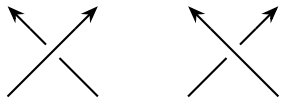
\begin{figure}[h]
		\centering
		\begin{tikzpicture}[knot gap=9, scale=0.6]
			
			\begin{scope}[xshift=5cm]
				\draw[thick, -{Stealth}] (0.5,-0.5) -- (-1,1);
				\draw[thick, -{Stealth}] (-1,-1) -- (1,1);
				\draw[thick, knot] (1,-1) -- (-0.5,0.5);
				\draw[thick, knot] (-1,-1) -- (0.5,0.5);
			\end{scope}
			
			\begin{scope}[xshift=9cm]
				\draw[thick, -{Stealth}] (0.5,-0.5) -- (-1,1);
				\draw[thick, -{Stealth}] (-1,-1) -- (1,1);
				\draw[thick, knot] (-1,-1) -- (0.5,0.5);
				\draw[thick, knot] (1,-1) -- (-0.5,0.5);
			\end{scope}
			
		\end{tikzpicture} 
		\caption{ A positive crossing (left) and a negative crossing (right)} \label{positive and negative crossing}
	\end{figure}

	\begin{definition}
		At any crossing in any link diagram $D$, we can perform an \underline{\textbf{$A$-smoothing}} or a \underline{\textbf{$B$-smoothing}}, shown in Figure \ref{A and B smoothing}. Performing $A$-smoothings on every crossing in the diagram results in an arrangement of circles known as the \underline{\textbf{$A$-state}}. The set of circles that make up the $A$-state are called the \underline{\textbf{$A$-circles}}. 
		We then can consider the associated \underline{\textbf{$A$-state graph}} of the diagram $D$: Every $A$-circle in $D$ corresponds to a vertex in the $A$-state graph, and every crossing in $D$ corresponds to an edge in the graph, as in Figure \ref{reduced A-state graph example}. The \underline{\textbf{reduced $A$-state graph}} is the $A$-state graph with duplicate edges removed. 
	\end{definition}
	
	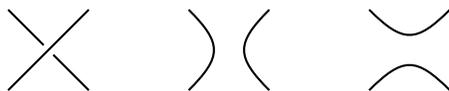
\begin{figure}[h]
    \centering
    \begin{tikzpicture}[every path/.style={thick}, every
node/.style={transform shape, knot crossing, inner sep=2.75pt}, scale=0.6]
    \node (tl) at (-1, 1) {};
    \node (tr) at (1, 1) {};
    \node (bl) at (-1, -1) {};
    \node (br) at (1, -1) {};
    \node (c) at (0,0) {};

    \draw (bl) -- (tr);
    \draw (br) -- (c);
    \draw (c) -- (tl);
    
    \begin{scope}[xshift=4cm]
    \node (tl) at (-1, 1) {};
    \node (tr) at (1, 1) {};
    \node (bl) at (-1, -1) {};
    \node (br) at (1, -1) {};

    \draw (bl) .. controls (bl.8 north east) and (tl.8 south east) .. (tl);
    \draw (br) .. controls (br.8 north west) and (tr.8 south west) .. (tr);
    \end{scope}

    \begin{scope}[xshift=8cm]
    \node (tl) at (-1, 1) {};
    \node (tr) at (1, 1) {};
    \node (bl) at (-1, -1) {};
    \node (br) at (1, -1) {};

    \draw (bl) .. controls (bl.8 north east) and (br.8 north west) .. (br);
    \draw (tl) .. controls (tl.8 south east) and (tr.8 south west) .. (tr);
    \end{scope}

    \end{tikzpicture}
    \caption{ A crossing (left), its $A$-smoothing (middle), and $B$-smoothing (right)}
    \label{A and B smoothing}
\end{figure}
	

\begin{figure}[h]
    \centering
    \begin{tikzpicture}[every path/.style={thick}, every
node/.style={transform shape, knot crossing, inner sep=2pt}]

\begin{scope}[scale=0.75]
 
    \begin{scope}[xshift=0cm, scale=0.8]
    \node (tlo) at (-6.25,2.25) {};
    \node (tro) at (-3.75,2.25) {};
    \node (so) at (-5,1.25) {};
    \node (mlo) at (-5.5,0) {};
    \node (mro) at (-4.5, 0) {};
    \node (bo) at (-5,-1.25) {};
    \draw (bo.center) .. controls (bo.4 north west) and (mlo.4 south west) ..
    (mlo.center);
    \draw  (bo) .. controls (bo.4 north east) and (mro.4 south east).. (mro);
    \draw  [-{Stealth}] (mlo) .. controls (mlo.8 north west) and (so.3 south west) ..
    (so);
    \draw [-{Stealth}] (mlo.center) .. controls (mlo.8 north east) and (mro.2 north
    west) .. (mro);
    \draw  (mro.center) .. controls (mro.4 north east) and (so.8 south east) ..
    (so.center);
    \draw (mro.center) .. controls (mro.8 south west) and (mlo.3 south east) ..
    (mlo);
    \draw (so) .. controls (so.4 north east) and (tro.8 south west) .. (tro);
    \draw  (so.center) .. controls (so.8 north west) and (tlo.8 south
    east) .. (tlo.center);
    \draw (tlo.center) .. controls (tlo.16 north west) and (tro.16 north
    east) .. (tro);
    \draw (bo.center) .. controls (bo.16 south east) and (tro.16 south east) ..
    (tro.center);
    \draw (bo) .. controls (bo.16 south west) and (tlo.16 south
    west) .. (tlo);
    \draw (tlo) .. controls (tlo.4 north east) and (tro.4 north
    west) .. (tro.center);
    \end{scope}
    
    \begin{scope}[xshift=-0.75cm, scale=0.8]
    \node (tl) at (-1.25,2.25) {};
    \node (tltl) at (-1.3,2.3) {};
    \node (tlbl) at (-1.3,2.2) {};
    \node (tltr) at (-1.2,2.3) {};
    \node (tlbr) at (-1.2,2.2) {};
    \node (tr) at (1.25,2.25) {};
    \node (trtl) at (1.2,2.3) {};
    \node (trbl) at (1.2,2.2) {};
    \node (trtr) at (1.3,2.3) {};
    \node (trbr) at (1.3,2.2) {};
    \node (stl) at (-0.05,1.3) {};
    \node (str) at (0.05,1.3) {};
    \node (sbl) at (-0.05,1.2) {};
    \node (sbr) at (0.05,1.2) {};
    \node (ml) at (-0.5,0) {};
    \node (mltl) at (-0.55,0.05) {};
    \node (mlbl) at (-0.55,-0.05) {};
    \node (mltr) at (-0.45,0.05) {};
    \node (mlbr) at (-0.45,-0.05) {};
    \node (mr) at (0.5, 0) {};
    \node (mrtl) at (0.45, 0.05) {};
    \node (mrbl) at (0.45, -0.05) {};
    \node (mrtr) at (0.55, 0.05) {};
    \node (mrbr) at (0.55, -0.05) {};
    \node (b) at (0,-1.25) {};
    \node (btl) at (-0.05, -1.2) {};
    \node (bbl) at (-0.05, -1.3) {};
    \node (btr) at (0.05, -1.2) {};
    \node (bbr) at (0.05, -1.3) {};
    
    \draw [gray] (btl.center) .. controls (btl.4 north west) and (mlbl.4 south west) ..
    (mlbl.center);
    \draw [gray] (btr.center) .. controls (btr.4 north east) and (mrbr.4 south east)
    .. (mrbr.center);
    \draw [gray] (mltl.center) .. controls (mltl.8 north west) and (sbl.3 south west) ..
    (sbl.center);
    \draw [green] (mltr.center) .. controls (mltr.4 north east) and (mrtl.4 north
    west) .. (mrtl.center);
    \draw [gray] (mrtr.center) .. controls (mrtr.4 north east) and (sbr.8 south east) ..
    (sbr.center);
    \draw [green] (mrbl.center) .. controls (mrbl.4 south west) and (mlbr.4 south east) ..
    (mlbr.center);
    \draw [blue] (str.center) .. controls (str.8 north east) and (trbl.8 south west) .. (trbl.center);
    \draw [blue] (stl.center) .. controls (stl.8 north west) and (tlbr.8 south
    east) .. (tlbr.center);
    \draw [magenta] (tltl.center) .. controls (tltl.16 north west) and (trtr.16 north
    east) .. (trtr.center);
    \draw [blue] (bbr.center) .. controls (bbr.16 south east) and (trbr.16 south east) ..
    (trbr.center);
    \draw [blue] (bbl.center) .. controls (bbl.16 south west) and (tlbl.16 south
    west) .. (tlbl.center);
    \draw [magenta] (tltr.center) .. controls (tltr.4 north east) and (trtl.4 north
    west) .. (trtl.center);
    
    \draw [blue] (stl.center) -- (str.center) {};
    \draw [gray] (sbl.center) -- (sbr.center) {};
    \draw [gray] (mltl.center) -- (mlbl.center) {};
    \draw [green] (mltr.center) -- (mlbr.center) {};
    \draw [green] (mrtl.center) -- (mrbl.center) {};
    \draw [gray] (mrtr.center) -- (mrbr.center) {};
    \draw [magenta] (tltl.center) -- (tltr.center) {};
    \draw [blue] (tlbl.center) -- (tlbr.center) {};
    \draw [magenta] (trtl.center) -- (trtr.center) {};
    \draw [blue] (trbl.center) -- (trbr.center) {};
    \draw [gray] (btl.center) -- (btr.center) {};
    \draw [blue] (bbl.center) -- (bbr.center) {};
    \end{scope}

    
    \begin{scope}[xshift=2cm, scale =0.8]
    \node (magenta) at (0, 2.75) {};
    \node (m) at (0, 3) {};
    \node (blue) at (0, 1) {};
    \node (c) at (0, .75) {};
    \node (cr) at (0.25, 0.75) {};
    \node (orange) at (1.5, 0.75) {};
    \node (o) at (1.75, 0.75) {};
    \node (ou) at (1.75, 0.5) {};
    \node (ol) at (1.5, 0.75) {};
    \node (green) at (1.75, -1) {};
    \node (g) at (1.75, -1.25) {};

     \filldraw [color=magenta, fill=magenta!15] (m) circle[radius=0.32cm];
    \filldraw [color=blue, fill=blue!15] (c) circle[radius=0.32cm];
    \filldraw [color=gray, fill=gray!20] (o) circle[radius=0.32cm];
    \filldraw [color=green, fill=green!15] (g) circle[radius=0.32cm];
    \draw (blue) .. controls (blue.3 north east) and (magenta.3 south east) .. (magenta);
    \draw (blue) .. controls (blue.3 north west) and (magenta.3 south west) .. (magenta);
    \draw (orange) .. controls (orange.3 north west) and (cr.3 north east) .. (cr);
    \draw (orange) .. controls (orange.3 south west) and (cr.3 south east) .. (cr);

    \draw (green) .. controls (green.3 north west) and (ou.3 south west) .. (ou);
    \draw (green) .. controls (green.3 north east) and (ou.3 south east) .. (ou);
    \end{scope}
    
    \begin{scope}[xshift=4.65cm, scale=0.8]
    \node (magenta) at (0, 2.75) {};
    \node (m) at (0, 3) {};
    \node (cyan) at (0, 1) {};
    \node (c) at (0, .75) {};
    \node (cr) at (0.25, 0.75) {};
    \node (orange) at (1.5, 0.75) {};
    \node (o) at (1.75, 0.75) {};
    \node (ou) at (1.75, 0.5) {};
    \node (green) at (1.75, -1) {};
    \node (g) at (1.75, -1.25) {};
    
   \filldraw [color=magenta, fill=magenta!15] (m) circle[radius=0.32cm];
    \filldraw [color=blue, fill=blue!15] (c) circle[radius=0.32cm];
    \filldraw [color=gray, fill=gray!20] (o) circle[radius=0.32cm];
    \filldraw [color=green, fill=green!15] (g) circle[radius=0.32cm];
    \draw (cyan) -- (magenta);
    \draw (orange) -- (cr);
    \draw (green) -- (ou);
    \end{scope}

\end{scope}
\end{tikzpicture} 
\caption{ (Left to right:) A positive link diagram, its $A$-circles, its $A$-state graph, and its reduced $A$-state graph }
\label{reduced A-state graph example}
\end{figure}
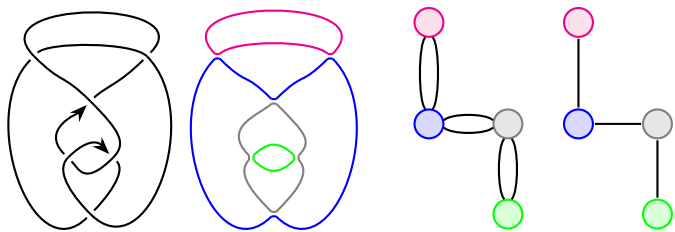

	Throughout this paper we assume we are working with non-split links and we use the following notation. 
	\begin{itemize}
		\item $c(D)$ is the number of crossings in diagram $D$
		\item $n(D)$ is the number of link components
		\item $A_D$ is the number of $A$-circles 
		\item $B_D$ is the number of $B$-circles
		\item $\nabla$ is the Conway polynomial and $\nabla_L$ is the Conway polynomial of link $L$
		\item $V$ is the Jones polynomial and $V_L$ is the Jones polynomial of link $L$
	\end{itemize}

	\begin{remark}
		In a positive diagram, $A$-smoothings are equivalent to smoothing according to Seifert's algorithm. So in a positive diagram, its $A$-circles are exactly the same as its Seifert circles.
	\end{remark}
	
	\begin{definition}
		The \underline{\textbf{cyclomatic number}} $p_1(G)$ of a connected graph $G$ is $p_1(G)= \#\text{edges} - \#\text{vertices} + 1$. The subscript $1$ refers to the fact that this is the first Betti number of the graph. 
	\end{definition}
	
	\begin{theorem}\label{second Jones coeff counts holes}(Stoimenow)\cite{Stoimenow} 
		
		Let $L$ be a positive link with positive diagram $D$. Then the second coefficient $V_1$ of the Jones polynomial satisfies:
		$$(-1)^{n(L)-1}V_1=\#(\text{Seifert circles})-1-\#(\text{pairs of Seifert circles that share at least one crossing}).$$ 
		
	\end{theorem}
	
	So for any positive diagram $D$, the absolute value of the second Jones coefficient is exactly the cyclomatic number of the reduced $A$-state graph of $D$. Hence any two positive diagrams $D$ and $D'$ of the same link $L$ will have reduced $A$-state graphs with idenitical cyclomatic numbers, and we can define the cyclomatic number of a positive link as in \cite{PS2020} and \cite{kegel2023khovanovhomologypositivelinks}.

	\begin{definition}
		The \underline{\textbf{cyclomatic number}} $p_1(L)$ of a positive link is the cyclomatic number of the reduced $A$-state graph of any positive diagram of $L$. 
	\end{definition}
	
	\begin{remark}
		Observe that if $G$ is a planar graph, then $p_1(G)$ is the number of interior faces. The reduced $A$-state graph of positive link diagram will always be planar. So the cyclomatic number $p_1(L)$ of a positive link $L$ counts the number of interior faces in the reduced $A$-state graph of any positive diagram $D$ of the link. 
	\end{remark}
	
	\begin{definition}
		A link diagram is \underline{\textbf{reduced}} if it has no nugatory crossings. 
	\end{definition}
	
	\begin{definition}
		A \underline{\textbf{type $j$ diagram}} is a positive, reduced diagram of a link $L$ with cyclomatic number $p_1(L)=j$.  Its second Jones coefficient is equal to $\pm j$. Figure \ref{reduced A-state graph example} shows a type $0$ diagram, and a type $1$ diagram is shown in Figure \ref{fig:Burdened type 1 example}. 
	\end{definition}

	\begin{figure}[h]
    \centering
  \begin{tikzpicture}
      \begin{scope}[scale=0.56, >=Stealth]

      \begin{scope}[scale=0.8, xshift=-6cm, rotate=45]
       \node (a) at (-3,0){};
       \node (b) at (-1,0){};
       \node (c) at (0,3) {};
       \node (d) at (0,1){};
       \node (e) at (1,0) {};
       \node (f) at (3,0){};
       \node (g) at (3,-3) {};
       \node (h) at (0, -3){};
       \node (i) at (-3,-3){};
       \node (j) at (-3, -2){};
       \node (k) at (-1, -2){};
       

       \draw [out=90, in=90, relative, decoration={markings, mark=at position 0.5 with {\arrow{<}}}, postaction={decorate}] 
       (a.center) to (c);
       \draw [out=90, in=90, relative, decoration={markings, mark=at position 0.5 with {\arrow{>}}}, postaction={decorate}] 
       (c.center) to (f);
       \draw [out=45, in=135, relative, decoration={markings, mark=at position 0.5 with {\arrow{}}}, postaction={decorate}] 
       (f.center) to (g.center);
       \draw [out=45, in=135, relative, decoration={markings, mark=at position 0.5 with {\arrow{<}}}, postaction={decorate}] 
       (g.center) to (h);
       \draw [out=45, in=135, relative, decoration={markings, mark=at position 0.5 with {\arrow{>}}}, postaction={decorate}] 
       (h.center) to (i.center);
       \draw [out=45, in=135, relative, decoration={markings, mark=at position 0.5 with {\arrow{}}}, postaction={decorate}] 
       (i.center) to (j.center);
       \draw [out=45, in=135, relative, decoration={markings, mark=at position 0.5 with {\arrow{>}}}, postaction={decorate}] 
       (j) to (a);

       \draw [ decoration={markings, mark=at position 0.5 with {\arrow{>}}}, postaction={decorate}] 
       (b) to (d.center);
       \draw [ decoration={markings, mark=at position 0.5 with {\arrow{>}}}, postaction={decorate}] 
       (e.center) to (d);
       \draw [out=-30, in=170, relative, decoration={markings, mark=at position 0.5 with {\arrow{<}}}, postaction={decorate}] 
       (h.center) to (e);
       \draw [decoration={markings, mark=at position 0.5 with {\arrow{<}}}, postaction={decorate}] 
       (h) to (k);
       \draw [out=40, in=140, relative, decoration={markings, mark=at position 0.5 with {\arrow{>}}}, postaction={decorate}] 
       (b.center) to (k.center);


       \draw [out=45, in=135, relative, decoration={markings, mark=at position 0.5 with {\arrow{>}}}, postaction={decorate}] 
       (a) to (b.center);
       \draw [out=45, in=140, relative, decoration={markings, mark=at position 0.5 with {\arrow{<}}}, postaction={decorate}] 
       (b) to (a.center);
       \draw [out=45, in=140, relative, decoration={markings, mark=at position 0.5 with {\arrow{>}}}, postaction={decorate}] 
       (d) to (c.center);
       \draw [out=45, in=135, relative, decoration={markings, mark=at position 0.5 with {\arrow{<}}}, postaction={decorate}] 
       (c) to (d.center);
       \draw [out=45, in=140, relative, decoration={markings, mark=at position 0.5 with {\arrow{>}}}, postaction={decorate}] 
       (f) to (e.center);
       \draw [out=45, in=135, relative, decoration={markings, mark=at position 0.5 with {\arrow{<}}}, postaction={decorate}] 
       (e) to (f.center);
       \draw [out=45, in=135, relative, decoration={markings, mark=at position 0.5 with {\arrow{>}}}, postaction={decorate}] 
       (j.center) to (k);
       \draw [out=45, in=135, relative, decoration={markings, mark=at position 0.5 with {\arrow{>}}}, postaction={decorate}] 
       (k.center) to (j);
       \end{scope}

      \begin{scope}[scale=0.8, xshift=2.5cm, rotate=45]
      	
       \node (a) at (-3,0){};
       \node (b) at (-1,0){};
       \node (c) at (0,3) {};
       \node (d) at (0,1){};
       \node (e) at (1,0) {};
       \node (f) at (3,0){};
       \node (g) at (3,-3) {};
       \node (h) at (0, -3){};
       \node (i) at (-3,-3){};
       \node (j) at (-3, -2){};
       \node (k) at (-1, -2){};
       

       \draw [red, thick, out=90, in=90, relative, decoration={markings, mark=at position 0.5 with {\arrow{}}}, postaction={decorate}] 
       (a.center) to (c);
       \draw [blue, thick, out=90, in=90, relative, decoration={markings, mark=at position 0.5 with {\arrow{}}}, postaction={decorate}] 
       (c.center) to (f);
       \draw [orange, thick, out=45, in=135, relative, decoration={markings, mark=at position 0.5 with {\arrow{}}}, postaction={decorate}] 
       (f.center) to (g.center);
       \draw [orange, thick, out=45, in=135, relative, decoration={markings, mark=at position 0.5 with {\arrow{}}}, postaction={decorate}] 
       (g.center) to (h);
       \draw [purp, thick, out=45, in=135, relative, decoration={markings, mark=at position 0.5 with {\arrow{}}}, postaction={decorate}] 
       (h.center) to (i.center);
       \draw [purp, thick, out=45, in=135, relative, decoration={markings, mark=at position 0.5 with {\arrow{}}}, postaction={decorate}] 
       (i.center) to (j.center);
       \draw [purp, thick, out=45, in=135, relative, decoration={markings, mark=at position 0.5 with {\arrow{}}}, postaction={decorate}] 
       (j) to (a);

       \draw [red, thick, decoration={markings, mark=at position 0.5 with {\arrow{}}}, postaction={decorate}] 
       (b) to (d.center);
       \draw [blue, thick, decoration={markings, mark=at position 0.5 with {\arrow{}}}, postaction={decorate}] 
       (e.center) to (d);
       \draw [orange, thick, out=-30, in=170, relative, decoration={markings, mark=at position 0.5 with {\arrow{}}}, postaction={decorate}] 
       (h.center) to (e);
       \draw [purp, thick, decoration={markings, mark=at position 0.5 with {\arrow{}}}, postaction={decorate}] 
       (h) to (k);
       \draw [purp, thick, out=40, in=140, relative, decoration={markings, mark=at position 0.5 with {\arrow{}}}, postaction={decorate}] 
       (b.center) to (k.center);


       \draw [red, thick, out=45, in=135, relative, decoration={markings, mark=at position 0.5 with {\arrow{}}}, postaction={decorate}] 
       (a) to (b.center);
       \draw [purp, thick, out=45, in=140, relative, decoration={markings, mark=at position 0.5 with {\arrow{}}}, postaction={decorate}] 
       (b) to (a.center);
       \draw [red, thick, out=45, in=140, relative, decoration={markings, mark=at position 0.5 with {\arrow{}}}, postaction={decorate}] 
       (d) to (c.center);
       \draw [blue, thick, out=45, in=135, relative, decoration={markings, mark=at position 0.5 with {\arrow{}}}, postaction={decorate}] 
       (c) to (d.center);
       \draw [orange, thick, out=45, in=140, relative, decoration={markings, mark=at position 0.5 with {\arrow{}}}, postaction={decorate}] 
       (f) to (e.center);
       \draw [blue, thick, out=45, in=135, relative, decoration={markings, mark=at position 0.5 with {\arrow{}}}, postaction={decorate}] 
       (e) to (f.center);
       \draw [green, thick, out=45, in=135, relative, decoration={markings, mark=at position 0.5 with {\arrow{}}}, postaction={decorate}] 
       (j.center) to (k);
       \draw [green, thick, out=45, in=135, relative, decoration={markings, mark=at position 0.5 with {\arrow{}}}, postaction={decorate}] 
       (k.center) to (j);
       \end{scope}
   

  \begin{scope}[scale = 1, xshift=7.75cm, yshift=0.5cm, rotate = 45]
  
       \node (a) at (-1,1){};
       \node (b) at (1,1){};
       \node (c) at (1,-1) {};
       \node (d) at (-1,-1){};
       \node (e) at (-2, -2) {};
        \draw[]
        (a) -- (b)
        (b) -- (c) 
        (d) -- (a)
        (c) -- (d);
        \draw[out=20, in=160, relative]
        (a) to (b)
        (b) to (c)
        (d) to (a)
        (d) to (e)
        (e) to (d);
        \draw[out=20, in=160, relative]
        (a) to (b)
        (b) to (c)
        (d) to (a);

        \filldraw[color=red, fill=red!10] 
        (a) circle (1/3);
     \filldraw[color=blue, fill=blue!10] 
        (b) circle (1/3);
     \filldraw[color=orange, fill=orange!10] 
        (c) circle (1/3);
     \filldraw[color=purp, fill=purp!10] 
        (d) circle (1/3);
     \filldraw[color=green, fill=green!10] 
        (e) circle (1/3);
        
       \end{scope}

  \begin{scope}[scale = 1, xshift=12cm, yshift=0.5cm, rotate = 45]
  
       \node (a) at (-1,1){};
       \node (b) at (1,1){};
       \node (c) at (1,-1) {};
       \node (d) at (-1,-1){};
       \node (e) at (-2, -2) {};
        \draw[]
        (a) -- (b)
        (b) -- (c) 
        (d) -- (a)
        (c) -- (d)
        (d) -- (e);

        \filldraw[color=red, fill=red!10] 
        (a) circle (1/3);
     \filldraw[color=blue, fill=blue!10] 
        (b) circle (1/3);
     \filldraw[color=orange, fill=orange!10] 
        (c) circle (1/3);
     \filldraw[color=purp, fill=purp!10] 
        (d) circle (1/3);
     \filldraw[color=green, fill=green!10] 
        (e) circle (1/3);
        
       \end{scope}

       
       \end{scope}
       
  \end{tikzpicture}
  \vspace{-5pt}
     \caption{A type $1$ diagram, its $A$-circles, $A$-state graph, and reduced $A$-state graph}
    \label{fig:Burdened type 1 example}
\end{figure}
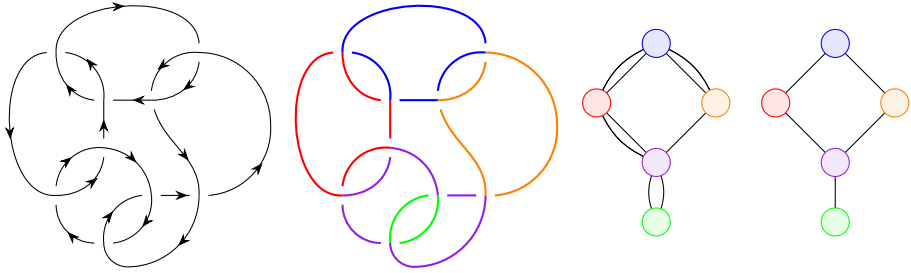

	\subsection{Balanced Diagrams}\label{Definitions}
	
	\begin{definition}\label{Balanced type j def}\label{Balanced type 1 def} 
		A \underline{\textbf{balanced type $\mathbf{j}$ diagram}} is a (non-split) type $j$ link diagram $D$ in which for each pair $v,w$ of $A$-circles, exactly one of the following is true:
		\begin{enumerate}
			\item $v$ and $w$ share $0$ crossings, 
			\item $v$ and $w$ share exactly $1$ crossing, and the edge in the reduced $A$-state graph of $D$ corresponding to that crossing is part of a cycle, or 
			\item $v$ and $w$ share exactly $2$ crossings, and the edge in the reduced $A$-state graph corresponding to those crossings is not part of a cycle. 
		\end{enumerate}
	\end{definition}
	
	The example in Figure \ref{reduced A-state graph example} is a balanced type $0$ diagram, and Figure \ref{fig:Balanced_type_1_example_2} shows a balanced type $1$ diagram.
	
	\input{Balanced_type_1_example_2}

We observe that for any type $j$ diagram $D$, a balanced type $j$ diagram $D_{bal}$ can be obtained by simply smoothing $m$ crossings, for some non-negative integer $m$. Observe that $D$ and $D_{bal}$ will have the same reduced $A$-state graphs. For example, compare the reduced $A$-state graphs shown in Figure \ref{fig:Burdened type 1 example} and Figure \ref{fig:Balanced_type_1_example_2}.

		\begin{definition}\label{def of burdening number}
		The number of crossings that must be smoothed away from a type $j$ diagram to produce a balanced type $j$ diagram is called the \underline{\textbf{burdening number}}. We denote the burdening number by $m$. 
	\end{definition}
	
	The burdening number is the least upper bound on the number of crossings that must be smoothed in order to obtain a balanced diagram of any type. It is also the greatest number of crossings that can be smoothed away to result in a balanced diagram of the same type. 
	
	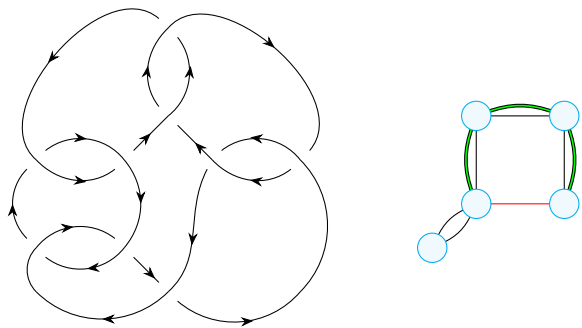
\begin{figure}[h]
    \centering
  \begin{tikzpicture}
      \begin{scope}[scale=0.65, >=Stealth]

      \begin{scope}[scale=0.9, xshift=-4cm]
       \node (a) at (-3,0){};
       \node (b) at (-1,0){};
       \node (c) at (0,3) {};
       \node (d) at (0,1){};
       \node (e) at (1,0) {};
       \node (f) at (3,0){};
       \node (g) at (3,-3) {};
       \node (h) at (0, -3){};
       \node (i) at (-3,-3){};
       \node (j) at (-3, -2){};
       \node (k) at (-1, -2){};
       

       \draw [out=90, in=90, relative, decoration={markings, mark=at position 0.5 with {\arrow{<}}}, postaction={decorate}] 
       (a.center) to (c);
       \draw [out=90, in=90, relative, decoration={markings, mark=at position 0.5 with {\arrow{>}}}, postaction={decorate}] 
       (c.center) to (f);
       \draw [out=45, in=135, relative, decoration={markings, mark=at position 0.5 with {\arrow{}}}, postaction={decorate}] 
       (f.center) to (g.center);
       \draw [out=45, in=135, relative, decoration={markings, mark=at position 0.5 with {\arrow{<}}}, postaction={decorate}] 
       (g.center) to (h);
       \draw [out=45, in=135, relative, decoration={markings, mark=at position 0.5 with {\arrow{>}}}, postaction={decorate}] 
       (h.center) to (i.center);
       \draw [out=45, in=135, relative, decoration={markings, mark=at position 0.5 with {\arrow{}}}, postaction={decorate}] 
       (i.center) to (j.center);
       \draw [out=45, in=135, relative, decoration={markings, mark=at position 0.5 with {\arrow{>}}}, postaction={decorate}] 
       (j) to (a);

       \draw [ decoration={markings, mark=at position 0.5 with {\arrow{>}}}, postaction={decorate}] 
       (b) to (d.center);
       \draw [ decoration={markings, mark=at position 0.5 with {\arrow{>}}}, postaction={decorate}] 
       (e.center) to (d);
       \draw [out=-30, in=170, relative, decoration={markings, mark=at position 0.5 with {\arrow{<}}}, postaction={decorate}] 
       (h.center) to (e);
       \draw [decoration={markings, mark=at position 0.5 with {\arrow{<}}}, postaction={decorate}] 
       (h) to (k);
       \draw [out=40, in=140, relative, decoration={markings, mark=at position 0.5 with {\arrow{>}}}, postaction={decorate}] 
       (b.center) to (k.center);


       \draw [out=45, in=135, relative, decoration={markings, mark=at position 0.5 with {\arrow{>}}}, postaction={decorate}] 
       (a) to (b.center);
       \draw [out=45, in=140, relative, decoration={markings, mark=at position 0.5 with {\arrow{<}}}, postaction={decorate}] 
       (b) to (a.center);
       \draw [out=45, in=140, relative, decoration={markings, mark=at position 0.5 with {\arrow{>}}}, postaction={decorate}] 
       (d) to (c.center);
       \draw [out=45, in=135, relative, decoration={markings, mark=at position 0.5 with {\arrow{<}}}, postaction={decorate}] 
       (c) to (d.center);
       \draw [out=45, in=140, relative, decoration={markings, mark=at position 0.5 with {\arrow{>}}}, postaction={decorate}] 
       (f) to (e.center);
       \draw [out=45, in=135, relative, decoration={markings, mark=at position 0.5 with {\arrow{<}}}, postaction={decorate}] 
       (e) to (f.center);
       \draw [out=45, in=135, relative, decoration={markings, mark=at position 0.5 with {\arrow{>}}}, postaction={decorate}] 
       (j.center) to (k);
       \draw [out=45, in=135, relative, decoration={markings, mark=at position 0.5 with {\arrow{>}}}, postaction={decorate}] 
       (k.center) to (j);
       \end{scope}



  \begin{scope}[scale = 0.9, xshift=4cm]
  
       \node (a) at (-1,1){};
       \node (b) at (1,1){};
       \node (c) at (1,-1) {};
       \node (d) at (-1,-1){};
       \node (e) at (-2, -2) {};
        \draw[]
        (a) -- (b)
        (b) -- (c) 
        (d) -- (a);
        \draw[red]
        (c) -- (d);
        \draw[out=20, in=160, relative]
        (a) to (b)
        (b) to (c)
        (d) to (a)
        (d) to (e)
        (e) to (d);
        \draw[out=20, in=160, relative, double=green]
        (a) to (b)
        (b) to (c)
        (d) to (a);

        \filldraw[color=cyan, fill=cyan!5] 
        (a) circle (1/3)
        (b) circle (1/3)
        (c) circle (1/3)
        (d) circle (1/3)       
        (e) circle (1/3);
        
       \end{scope}

       
       \end{scope}
       
  \end{tikzpicture}
    \caption{Smoothing the crossing that corresponds to the red edge will transform this type $1$ diagram into a balanced type $0$ diagram. Smoothing crossings corresponding to the three green edges will transform this type $1$ diagram into a balanced type $1$ diagram.  The burdening number of the diagram is $m=3$.}
    \label{fig:smoothing number vs burdening number}
\end{figure}

	\section{Structure of Main Argument}\label{structure of main argument} 
	
	From examining the Kauffman state sum model of the Jones polynomial, it becomes clear that for any link diagram of any link, the minimum degree of the Jones polynomial is bounded below by the degree of the contribution of the $A$-state \cite{Kauffman}. For positive diagrams in particular, the lowest degree term is contributed solely by the $A$-state, and can be neatly expressed in terms of the diagram's crossing number and number of $A$-circles. 
	
	\begin{equation}\label{min degree pos}
		\min \deg V_L = \frac{c(D) - A_D + 1}{2}
	\end{equation}
	
	Similarly, for any link diagram of any link, the maximum degree of the Jones polynomial is bounded above by the contribution of the $B$-state. For positive diagrams in particular, this bound can be neatly expressed in terms of the crossing number and number of $B$-circles. 
	
	\begin{equation}\label{standard bound}
		\max \deg V_L \leq c(D) + \frac{B_D -1 }{2}.
	\end{equation}
	
	\ref{min degree pos} and \ref{standard bound} are both specializations of a more general result appearing in Lickorish \cite{Lickorish}. 
	
	In this paper, we develop an upper bound on $c(D)+ \frac{B_D -1}{2}$ for links with type $1$ diagrams, essentially replacing all diagram-dependent quantities with diagram-independent quantities.

	In Section \ref{Balanced Diagrams and the Clasp Move} Theorem \ref{Bal type 1 B=n} we prove that in a \textit{balanced} type 1 diagram, the number of $B$-circles is equal to the number of link components. At the start of Section \ref{Type 1 Diagrams and the Burdening Number}, we prove an analogous result for an \textit{arbitrary} type 1 diagram. \\
	
	\noindent \textbf{Corollary \ref{bound on B-circles Bur}.} \textit{Let $D$ be a type 1 diagram with burdening number $m$ and $n$ link components. Then $B_D \leq n + 2m$. }\\

	At the end of Section \ref{Type 1 Diagrams and the Burdening Number} we find an expression for the burdening number of a type 1 diagram in terms of diagram-independent quantities and the length of the cycle in the diagram's reduced $A$-state graph.\\
	
	\noindent \textbf{Proposition \ref{m for coeff pm 1}.} \textit{Let $D$ be a type $1$ diagram of a link $L$. Then the burdening number $m$ can be expressed as 
		$$m = 4 \min \deg V_{L} - c(D) + k -2$$
		where $k$ is the length of the cycle in the reduced $A$-state graph of $D$. }\\

	Finally, in Section \ref{Conway Polynomial}, we find that the length of the cycle in a type $1$ diagram's reduced $A$-state graph is an invariant among all positive diagrams of the link.\\

	\noindent \textbf{Lemma \ref{lead coeff k-Bur}.} \textit{Let $D$ be a type $1$ diagram. Let $k$ be the length of the cycle in its reduced $A$-state graph. Then the leading coefficient of its Conway polynomial is $$\lead \coeff \nabla_D = \frac{k}{2}.$$}

	We now use these three statements to prove our main result. 

	\begin{theorem} \label{main result coeff 1}
		Let $L$ be a positive link with $n$ link components, Jones polynomial $V_L$, and Conway polynomial $\nabla_L$. If the second coefficient of $V_L$ is $\pm 1$, then  $$\max \deg V_L \leq 4 \min \deg V_L + \frac{n-1}{2} + 2\lead \coeff \nabla_L -2, $$
		where $\lead \coeff \nabla_L$ is the coefficient of the highest degree term of the Conway polynomial of link $L$.
	\end{theorem}
		
		\begin{proof}
			Let $L$ be a positive link with second Jones coefficient $\pm 1$. Let $D$ be a positive diagram of $L$. Then $D$ is a type $1$ diagram. Let $k$ be the length of the cycle in its reduced $A$-state graph. Then
			
				\begin{align*}
				\max \deg V_D &\leq c(D) + \frac{B_D - 1}{2} &&\text{ (by \ref{standard bound})}\\
				& \leq c(D) + \frac{n-1}{2} + m &&\text{ (by Corollary \ref{bound on B-circles Bur})}\\
				& = c(D) + \frac{n-1}{2}+ 4\min \deg V_L - c(D) + k - 2 &&\text{ (by Proposition \ref{m for coeff pm 1})}\\
				& = 4\min \deg V_L + \frac{n-1}{2} + 2\lead \coeff \nabla_L - 2 &&\text{ (by Lemma \ref{lead coeff k-Bur})}.
			\end{align*}
		\end{proof}
	
	Hence, it just remains to prove those three key pieces.

	\section{Balanced Diagrams and the Clasp Move}\label{Balanced Diagrams and the Clasp Move}
	
	In \cite{Buchanan_2022}, we found that in a balanced type $0$ diagram, the number of $B$-circles is equal to the number of link components. 
	
	\begin{theorem}\label{B=n for Balanced type 0}\cite{Buchanan_2022}
		Let $D$ be a balanced diagram of type $0$ with $n$ link components. Then $$B_D = n(D).$$
	\end{theorem}
	
	At the end of this section we prove that this is also true for balanced type $1$ diagrams:\\
	
		\noindent \textbf{Theorem \ref{Bal type 1 B=n}.} \textit{Let $D$ be a balanced diagram of type $1$ with $n$ link components. Then the number of $B$-circles in $D$ is equal to the number of link components: 
			$$B_D = n(D).$$}
	
We will prove Theorem \ref{Bal type 1 B=n} by relating every balanced type $1$ diagram to a balanced type $0$ diagram via a sequence of \textit{clasp moves}. This clasp move preserves both the number of link components and the number of $B$-circles, so then Theorem \ref{Bal type 1 B=n} naturally follows from Theorem \ref{B=n for Balanced type 0}.
	
	\begin{remark}\label{even cycles}
		For a positive diagram, the $A$-state graph is the same as the Seifert graph. Since the Seifert graph of any link diagram is bipartite, and bipartite graphs do not have any odd cycles, the $A$-state graph of any positive diagram has only even cycles \cite{Cromwell_2004}. Hence the reduced $A$-state graph of a type $j$ diagram for $j\geq 1$ contains only even cycles, and must contain a cycle of length $\geq 4$. 
	\end{remark}

	\input{clasp_example_1}
	
	\subsection{Motivating the Clasp Move}\label{motivating the clasp move}
	
	Throughout the following, whenever we refer to \say{an arc} of a diagram, we mean a portion of a strand that goes between two crossings, so an arc ends when it reaches any crossing, not just an undercrossing.

	In Figure \ref{fig:clasp example 1} we see a balanced diagram of type $1$ that we transform into a balanced diagram of type $0$ by adding a clasp. Let $A_1, A_2, A_3, A_4$ be the $A$-circles of this diagram, with corresponding vertices $v_1, v_2, v_3, v_4$ forming a cycle in the $A$-state graph. We take arcs in the diagram that belong to some $A$-circles $A_1$ and $A_3$, and interlock them. What happened in the $A$-state graph? All edges of the form $(v_i, v_3)$ are now of the form $(v_i, v_1)$, and we have added two copies of the new edge $(v_1, v_3)$. This clasping did not change the number of $A$-circles in $D$ (or the number of vertices of $G$), but it did change how the circles (and their corresponding edges in $G$) are arranged relative to one another. This is shown in a little more generality in Figure \ref{fig:clasp move 0}, which demonstrates how we can view this process as one $A$-circle swallowing the other. This kind of clasping clearly does not change the number of link components. But more interestingly, we can see in Figure \ref{fig:clasp move 1} that this does not change the number of $B$-circles either. By adding more clasps we can kill a hole in a reduced $A$-state graph (decrease the cyclomatic number), and preserve the number of link components and number of $B$-circles as we do so. 
	
	\input{clasp_move_0}
	\input{clasp_move_1}
	
	Since the $A$-state graph will tell us information to help us classify a diagram as balanced or not, we would like to be able to think about performing \textit{clasp moves} directly from the information contained in the graph, and not have to start with a diagram.
	
	\subsection{The Clasp Move}\label{the clasp move}
	
	\begin{definition}
		In a connected graph, a \underline{\textbf{cut edge}} is an edge whose deletion disconnects the graph. An edge is a cut edge if and only if it is not part of a cycle. 
	\end{definition}
	
	Let $D$ be a reduced positive link diagram with the following properties: Every cut edge in the reduced $A$-state graph of $D$ corresponds to exactly two crossings in $D$, and every cycle edge in the reduced $A$-state graph corresponds to exactly one crossing in $D.$
	
	Suppose there is a path $(v_1, v_2, v_3)$ in the reduced $A$-state graph that is part of a cycle, and that $v_2$ has degree $2$. Then there are arcs $a_1$ and $a_3$ in the diagram $D$ (that are part of the $A$-circles corresponding to vertices $v_1$ and $v_3$) that can be clasped together with positive crossings. We saw this in Figure \ref{fig:clasp example 1}.
	
	Now, suppose that: (1) $(v_1, v_2, v_3)$ is part of a cycle in the reduced $A$-state graph, (2) Cutting edges $(v_1, v_2)$ and $(v_2, v_3)$ disconnects the graph, and (3) this disconnects it so that the component containing $v_2$ is a tree. So, this means that $v_2$ may not have degree $2$ in the graph, but from the perspective of any vertex in the graph that is not part of that tree rooted at $v_2$, $v_2$ might as well have degree $2$ -- every path from $v_2$ to any other vertex (that is not part of the tree) must use edge $(v_1, v_2)$ or $(v_2, v_3)$. Then, as before, we can clasp together arcs in the link diagram corresponding to $v_1$ and $v_3$. This is shown in Figure \ref{fig:v2 cycle degree 2}, and will be called a \textit{clasp move}.

	\begin{definition} \label{claspable def}
		A reduced positive link diagram $D$ is \textbf{\underline{claspable}} if it satisfies the following: \begin{enumalph} 
			\item Every cut edge in the reduced $A$-state graph of $D$ corresponds to exactly two edges in the $A$-state graph of $D$ (and thus also corresponds to exactly two crossings in $D$)
			\item Every cycle edge in the reduced $A$-state graph corresponds to exactly one edge in the $A$-state graph (and thus to exactly one crossing in $D$)
			\item The reduced $A$-state graph contains a cycle $(v_1, v_2, v_3, \dots, v_1)$ such that: \label{cycle condition}
			\begin{enumalph}
				\item Edges $(v_1, v_2)$ and $(v_2, v_3)$ form a cut set in the graph
				\item Cutting those two edges disconnects the graph in such a way that the component containing $v_2$ is a tree.
			\end{enumalph}
		\end{enumalph}
		
	\end{definition}
	
	\begin{definition} \label{clasp move def}
		We perform a \textbf{\underline{clasp move}} on a claspable diagram $D$ (and its associated $A$-state graph $G$) by first locating arcs $a_1$ and $a_3$ that are part of $A$-circles corresponding to vertices $v_1$ and $v_3$ in the cycle described in condition (c) such that arcs $a_1$ and $a_3$ bound the same region in the link diagram. Then we pull arcs $a_1$ and $a_3$ towards each other, and clasp them together with positive crossings. In the $A$-state graph, this clasp move corresponds to transferring all edges incident to $v_3$ to be incident to $v_1$, and then adding in two copies of the edge $(v_1, v_3)$.
	\end{definition}

	\input{v2_cycle_degree_2}

	\begin{proposition} \label{leaves dont matter}
		Let $D$ be a claspable diagram, and let $D'$ be the diagram obtained by performing a clasp move. Let $G$ be the $A$-state graph of $D$, and let $G'$ be that of $D'$. Then $G$ and $G'$ have the same vertex set. If some vertex $v_i$ is incident to only one vertex in $G$, then $v_i$ is still incident to only one vertex in $G'$. 
	\end{proposition}
	
	\begin{proof}
		
		Suppose for contradiction that $v_i$ is incident to more than one vertex in $G'$. 
		The clasp move takes all edges in the $A$-state graph $G$ of the form $(v_3, v_j)$ and makes them edges of the form $(v_1, v_j)$ in $G'$, and all other edges remain exactly the same. 
		
		The only way a vertex can have a different set of neighbors in $G'$ than in $G$ is if the vertex is $v_3$, or if it has $v_3$ as a neighbor. Since by assumption $v_3$ is part of a cycle in the reduced $A$-state graph and $v_i$ is not, $v_i$ must be incident to $v_3$ in $G$. But then $v_i$ is only incident to $v_1$ in $G'$. 
	\end{proof}
	
	\begin{lemma}\label{clasp moves preserve B and n}
		Clasp moves do not change the number of link components or the number of $B$-circles in a positive diagram $D$: If $D'$ is obtained by performing a clasp move on $D$, then $B_{D'} = B_D$ and $n(D')=n(D)$. 
	\end{lemma}
	
	\begin{proof}
		Obviously, clasp moves do not change the number of link components in a diagram, and in Figure \ref{fig:clasp move 0} we saw that a clasp move preserves the number of $A$-circles. Now in Figure \ref{fig:clasp move 1}, we see that clasp moves also preserve the number of $B$-circles, regardless of whether the arcs belonged to the same $B$-circle or to different $B$-circles in the original diagram $D$.
	\end{proof}

	\subsection{Clasp Moves on a Balanced diagram of type $\mathbf{1}$}\label{clasp moves on bal type 1}
	
	\begin{proposition}\label{type 1 bal is claspable}
		Every type $1$ balanced diagram $D$ is claspable. 
	\end{proposition}
	
	\begin{proof}
		
		This follows immediately from Definition \ref{Balanced type 1 def}.
	\end{proof}

	\begin{theorem}\label{Bal type 1 B=n}
		Let $D$ be a balanced diagram of type $1$. Then the number of $B$-circles in $D$ is equal to the number of link components: 
		$B_D = n(D).$
	\end{theorem}
	
	\begin{proof}
		Let $G_A'$ be the reduced $A$-state graph of $D$. Since $D$ is a balanced type $1$ diagram, $G_A'$ contains exactly one cycle. Let $k$ be the length of that cycle. By \ref{even cycles}, $k$ is even. 
		
		We know that $D$ is claspable, and we can perform a clap move at any $(v_1, v_2, v_3)$ path segment of the cycle. If $k=4$, then after the clasp move we have a balanced diagram of type $0$, as seen in Figure \ref{fig:clasp example 1}. If instead $k\geq 6$, then after the clasp move we have a new balanced diagram of type $1$, whose sole cycle in the reduced $A$-state graph is of length $k-2$. This is seen in Figure \ref{fig:clasping k-balanced}. By Proposition \ref{type 1 bal is claspable}, this new diagram is also claspable. 
		
		Then by induction on $k$, for any balanced type $1$ diagram $D$ we can perform $\frac{k-2}{2}$ clasp moves and obtain a balanced type $0$ diagram $D'$. 
		
		Thus 
		\begin{align*}
			B_D &= B_{D'} \text{ by Lemma \ref{clasp moves preserve B and n}},\\
			& = n(D') \text{ by Theorem \ref{B=n for Balanced type 0}, since $D'$ is balanced type $0$},\\
			& = n(D) \text{ by Lemma \ref{clasp moves preserve B and n}}.
		\end{align*}
		
	\end{proof}
	
	\begin{figure}[h]
    \centering
  \begin{tikzpicture}
      \begin{scope}[scale=0.9, >=Stealth]
       
  \begin{scope}[scale = 0.6, xshift=-8cm]
  
       \node (a) at (-3,1){};
       \node (b) at (-1,1){};
       \node (c) at (1,1) {};
       \node (d) at (3,1){};
       \node (e) at (3, -1) {};
       \node (f) at (1, -1) {};
       \node (g) at (-1, -1) {};
       \node (h) at (-3, -1) {};
        \draw[]
        (a) -- (b)
        (b) -- (c) 
        (c) -- (d)
        (d) -- (e)
        (e) -- (f)
        (f) -- (g)
        (g) -- (h);
        \draw[dashed]
        (h) -- (a);
        
        \filldraw[color=cyan, fill=cyan!5] 
        (a) circle (1/3)
        (b) circle (1/3)
        (c) circle (1/3)
        (d) circle (1/3)       
        (e) circle (1/3)
        (f) circle (1/3)
        (g) circle (1/3)
        (h) circle (1/3); 
        
       \end{scope}

\begin{scope}[scale = 0.6, xshift=0cm]
  
       \node (a) at (-3,1){};
       \node (b) at (-1,1){};
       \node (c) at (1,1) {};
       \node (d) at (3,1){};
       \node (e) at (3, -1) {};
       \node (f) at (1, -1) {};
       \node (g) at (-1, -1) {};
       \node (h) at (-3, -1) {};
        \draw[]
        (c) -- (f)
        (a) -- (b)
        (b) -- (c) 
        (c) -- (d)
        (f) -- (g)
        (g) -- (h);
        \draw[dashed]
        (h) -- (a);
        \draw[out=20, in=160, relative]
        (c) to (d)
        (e) to (c)
        (c) to (e);
        
        \filldraw[color=cyan, fill=cyan!5] 
        (a) circle (1/3)
        (b) circle (1/3)
        (c) circle (1/3)
        (d) circle (1/3)       
        (e) circle (1/3)
        (f) circle (1/3)
        (g) circle (1/3)
        (h) circle (1/3); 
        
       \end{scope}

\begin{scope}[scale = 0.6, xshift=8cm]
  
       \node (a) at (-3,1){};
       \node (b) at (-1,1){};
       \node (c) at (1,1) {};
       \node (d) at (3,1){};
       \node (e) at (3, -1) {};
       \node (f) at (1, -1) {};
       \node (g) at (-1, -1) {};
       \node (h) at (-3, -1) {};
        \draw[]
        (b) -- (g)
        (a) -- (b)
        (b) -- (c) 
        (c) -- (d)
        (g) -- (h);
        \draw[dashed]
        (h) -- (a);
        \draw[out=20, in=160, relative]
        (b) to (c)
        (c) to (d)
        (e) to (c)
        (c) to (e)
        (f) to (b)
        (b) to (f);
        
        \filldraw[color=cyan, fill=cyan!5] 
        (a) circle (1/3)
        (b) circle (1/3)
        (c) circle (1/3)
        (d) circle (1/3)       
        (e) circle (1/3)
        (f) circle (1/3)
        (g) circle (1/3)
        (h) circle (1/3); 
        
       \end{scope}
       
       \end{scope}
       
  \end{tikzpicture}
  \vspace{-8pt}
    \caption{Performing successive clasp moves}
    \label{fig:clasping k-balanced}
\end{figure}
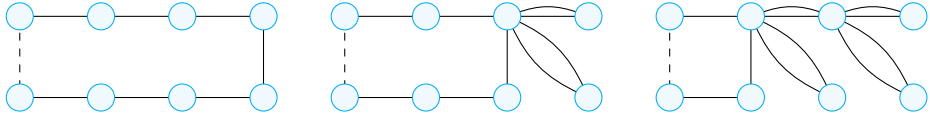

	\section{Type 1 Diagrams and the Burdening Number}\label{Type 1 Diagrams and the Burdening Number}
	
	In this section, we generalize Theorem \ref{Bal type 1 B=n}. Recall from Defintion \ref{def of burdening number} the concept of the \textit{burdening number} $m$: any arbitrary type $j$ diagram can be transformed into a balanced type $j$ diagram by performing a sequence of $m$ smoothings on select crossings. We now track the effect of those smoothings on the number of $B$-circles in a diagram. 
	
	\begin{corollary} [Corollary to Theorem \ref{Bal type 1 B=n}]
		\label{bound on B-circles Bur}
		Let $D$ be a type $1$ diagram with burdening number $m$.  Then 
		$$B_D \leq n(D) + 2m.$$
	\end{corollary}
	
	\begin{proof}
		By definition of burdening number, there exist $m$ crossings in $D$ such that smoothing those $m$ crossings results in balanced type $1$ diagram $D'$. Each smoothing changes the number of $B$-circles in the diagram by $\pm 1$ and changes the number of link components by $\pm 1$. So $|B_D - B_{D'}| \leq m $ and $|n(D) - n(D')| \leq m$. Then by Theorem  \ref{Bal type 1 B=n}, we can say that 
		\begin{align*}
			B_D &\leq B_{D'} + m \\
			& = n(D') + m \\
			&\leq n(D) + 2m.
		\end{align*}
	\end{proof}

	\begin{proposition}\label{m for coeff pm 1}
		Let $D$ be a type $1$ diagram of a link $L$. Then the burdening number $m$ can be expressed as 
		$$m = 4 \min \deg V_{L} - c(D) + k -2$$
		where $k$ is the length of the cycle in the reduced $A$-state graph of $D$. 
	\end{proposition}
	
	\begin{proof}
		
		We first consider a balanced type $1$ diagram $D$. Let $G_A$ be the $A$-state graph, and $G_A'$ the reduced $A$-state graph. Let $k$ be length of the cycle in $G_A'$. Then the number of edges in $G_A'$ is $A_D - 1 + 1 = A_D$. Since $D$ is balanced, every edge in $G_A'$ that is part of a cycle corresponds to one crossing in $D$, and every edge in $G_A'$ that is not part of a cycle corresponds to two crossings in $D$. So the number of crossings in $D$ in 
		
		\begin{equation}\label{number of crossings in k-Balanced}
			c(D) = 2(A_D - k) + k = 2A_D - k.
		\end{equation}
		
		Now let $D$ be an arbitrary type $1$ diagram whose reduced $A$-state graph has a single cycle of length $k$. The burdening number $m$ is the number of crossings in $D$ that must be smoothed to obtain a balanced type $1$ diagram $D'$ with the same reduced $A$-state graph. Then $c(D) = c(D') + m$ and $A_{D} = A_{D'}$. Thus 
		\begin{equation}\label{number of crossings in k-Burdened}
			c(D)=2A_D - k + m. 
		\end{equation}
		
		Since $D$ is a positive diagram, we have from \ref{min degree pos} that $4\min \deg V_L = 2(c(D) - A_D + 1)$. It follows from \ref{number of crossings in k-Burdened} that 
		\begin{align*}
			m & = c(D) - 2A_D + k \\
			& = 4\min \deg V_{L} - c(D) + k -2.
		\end{align*}
	\end{proof}
	
	\section{Conway Polynomial}\label{Conway Polynomial}
	
	Looking back at the proof of Theorem \ref{main result coeff 1}, we now have that the Jones polynomial of a positive link with a type 1 diagram must satisfy $\max\deg V \leq 4\min\deg V + \frac{n-1}{2} + k -2$. We have eliminated most of the diagram-dependent quantities from this bound, but one still remains: $k$, the length of the cycle in the reduced $A$-state graph of a type 1 diagram of the link. In this section we develop Lemma \ref{lead coeff k-Bur}, which says that the length of the cycle in the reduced $A$-state graph of a type $1$ diagram $D$ is equal to twice the leading coefficient of its Conway polynomial, and so is an invariant of the link.

	In Subsection \ref{on the leading coeff conway positive}, we gather whatever tools we can about positive and positive links, in preparation to describe the leading Conway coefficient of a generic positive link. We build up resources here that will always allow us to \say{prune away} any leaves (or other foliage), and instead find that we can compute the leading Conway coefficient just from the parts of the diagram that correspond to cycle edges in the reduced $A$-state graph. Everything in this subsection actually applies to any positive diagram, not just type 1 diagrams. 
	Subsection \ref{lead conway coeff coeff 1 section} specializes to the case of type 1 diagrams, and we prove Lemma \ref{lead coeff k-Bur}.
	
	\subsection{On the Leading Coefficient of the Conway Polynomial of a Positive Link}\label{on the leading coeff conway positive}

	In this subsection we exploit the Alexander-Conway Skein Relation to show that any two positive diagrams with the same reduced $A$-state graph will also have the same leading Conway coefficient. Since this part is not specific to type 1 diagrams, it will be useful for future arguments involving type $j$ diagrams for arbitrarily large values of $j$. 
	
	We use $\chi(D)$ to denote the Euler characteristic of the Seifert surface obtained from $D$ via Seifert's algorithm. Thus $\chi(D)$ is equal to the number of Seifert circles minus the crossing number of the diagram. $\chi(L)$ is the maximum value of the Euler characteristic of any Seifert surface of the link, meaning it is the Euler characteristic of a genus-minimizing Seifert surface of the link, and in general $\chi(D) \leq \chi(L)$. 
 
	Recall the Alexander-Conway Skein Relation: For diagrams $D_+, D_-,$ and $D_0$, which are identical except at one crossing which is positive in $D_+$, negative in $D_-$, and smoothed according to orientation in $D_0$, we have $$\nabla_+ - \nabla_- = z\nabla_0$$ where $\nabla_+, \nabla_-, \nabla_0$ are the Conway polynomials of $D_+, D_-, D_0$ (respectively). 
	
	The following propositions relate the Conway polynomial and Jones polynomial of a positive or almost-positive diagram $D$ to $\chi(D)$ and $\chi(L)$. 
	
	\begin{proposition}[Cromwell, \cite{Cromwell}]\label{Cromwell pos Conway degree}
		Let $L$ be a positive link with positive diagram $D$. Then 
		$$1- \chi(D) = 1-\chi(L) = \max\deg \nabla_L = 2\min\deg V_L.$$

	\end{proposition}

	\begin{proposition}[Stoimenow, \cite{Stoimenow2014MinimalGA}]\label{Stoimenow almost pos Conway degree min genus}
		Let $L$ be an almost-positive link. Then 
		$$1-\chi(L) = \max \deg \nabla_L = 2\min \deg V_L$$ 
	\end{proposition} 
	
	The next two propositions involve considering two types of almost-positive diagrams: one type in which the negative crossing and some positive crossing both connect the same pair of Seifert circles, and another type in which no positive crossing connects the same pair of Seifert circles as the negative crossing. Discussion of these two separate situations appears in the work of Feller, Lewark, and Lobb (\cite{Feller_2022}, notions of \textit{parallel crossings} and \textit{type I and type II diagrams}); Ito and Stoimenow (\cite{Stoimenow, StI} notions of \textit{Seifert equivalent crossings} and \textit{type I and type II diagrams}, and \textit{good and bad crossings} and \textit{good successively $k$-almost positive diagrams}); and Tagami (\cite{Tagami_2014}). We recall that for positive crossings, performing an $A$-smoothing is the same as smoothing according to Seifert's algorithm, so in a positive diagram the $A$-state circles are exactly the same as the Seifert circles.
	
	\begin{lemma}[Stoimenow, \cite{StI}]\label{stoimenow's lemma}
		Let $L$ be a link represented by an almost-positive diagram $D$ with negative crossing $q$. If $D$ is of type $I$ (there is no other crossing $p$ which connects the same pair of Seifert circles as $q$), then $\chi(L) = \chi(D)$. If $D$ is of type $II$ (there is another crossing $p$ which connects the same pair of Seifert circles as $q$), then $\chi(L) - 2  = \chi(D)< \chi (L)$. 
	\end{lemma}
	
	\begin{proposition}\label{another crossing doesnt affect conway degree}
		Let $D_+$ be a positive diagram with one distinguished crossing $q$, and let $D_-$ be the result of making $q$ negative. If there is another crossing in $D_+$ connecting the same two $A$-circles as $q$, then $$\max \deg\nabla_+ = \max \deg \nabla_- + 2$$ where $\nabla_+$ is the Conway polynomial of $D_+$ and $\nabla_-$ is the Conway polynomial of $D_-$. 
	\end{proposition}
	
	\begin{proof}
		
		For such $D_+$ and $D_-$, 
		$$	\max\deg \nabla_{+}  = 1 - \chi(D_+) = 1 - \chi(D_-) = 1 - \Big(\chi(L_-) - 2 \Big) = 1 - \chi(L_-) + 2 = \max \deg \nabla_{-} + 2
		$$
		where $L_-$ is the link represented by diagram $D_-$. The first equality is given by Proposition \ref{Cromwell pos Conway degree}, the third by Lemma \ref{stoimenow's lemma}, and the last equality follows from Proposition \ref{Cromwell pos Conway degree} if $L_-$ is a positive link or from Proposition \ref{Stoimenow almost pos Conway degree min genus} if $L_-$ is an almost-positive link.

	\end{proof}

	The preceding argument came from an anonymous reviewer. 
	
	As we will see, what this means for us is that if two crossings connect the same pair of $A$-circles, we can smooth one of them away and not change the degree or leading coefficient of the Conway polynomial. This also appears in the work of Stoimenow and Ito \cite{StI}.

	\begin{corollary}\label{another crossing doesnt affect leading term}
		Let $D_+$ be a positive diagram with one distinguished crossing $q$, let $D_-$ be the result of making $q$ negative, and let $D_0$ be the result of smoothing $q$. If there is another crossing in $D_+$ connecting the same two $A$-circles as $q$, then $$\lead \term \nabla_+ = z \lead \term \nabla_0.$$
		
	\end{corollary}
	
\begin{proof}
	
	The Alexander-Conway skein relation tells us that $\nabla_+ - \nabla_- = z\nabla_0$. Since $\max\deg \nabla_+ > \max\deg \nabla_-$ by Proposition \ref{another crossing doesnt affect conway degree}, we must have $\max\deg \nabla_+ = \max\deg(z\nabla_0)$. Furthermore, $\lead\term \nabla_+ = z \lead \term \nabla_0$.
\end{proof}

	\begin{lemma}\label{smooth away, keep lead coeff}
		Let $D$ be a positive diagram. Let $D_m$ be the result of adding $m$ crossings to $D$ such that for every intermediate diagram $D_i$ (for $0 < i\leq m$), the reduced $A$-state graph of $D_i$ is exactly the reduced $A$-state graph of $D$. Then $$\lead \term \nabla_{D_m} = z^m \lead \term \nabla_{D}.$$
	\end{lemma}
	
	\begin{proof}
		
		We proceed by induction on $m$. If $m = 1$, then we have added a single crossing $q$ to $D$ to create $D_1$, and we have not changed the underlying reduced $A$-state graph structure. Therefore, there exists another crossing $p$ in $D_1$ that connects the same pair of $A$-circles as $q$. Letting $D_1$ play the role of $D_+$, and $D$ play the role of $D_0$, we have by Corollary \ref{another crossing doesnt affect leading term} that 
		$$\lead \term \nabla_{D_1} = z \lead \term \nabla_{D}.$$ 
		Now assume there is some value of $m$ for which the statement holds. Consider diagram $D_{m+1}$, the result of adding $m+1$ crossings to $D$. This can also be viewed as the result of adding $1$ crossing to some $D_m$ while preserving the underlying graph structure. Therefore, by the same argument as in the base case, 
		$$\lead \term \nabla_{D_{m+1}} = z \lead \term \nabla_{D_m}.$$ By inductive hypothesis, $\lead \term \nabla_{D_m} = z^m \lead \term \nabla_{D},$ so in total 
		$$\lead \term \nabla_{D_{m+1}} = z \Big( z^m \lead \term \nabla_{D} \Big) = z^{m+1} \lead \term \nabla_{D}.$$
	\end{proof}

	\begin{lemma}\label{contract cut edges, same result}
		Let $D$ be a positive link diagram with reduced $A$-state graph $G$. Let $G'$ be the result of contracting all cut edges in $G$. Then for any positive link diagram $D'$ whose reduced $A$-state graph is $G'$, 
		$$\lead \coeff \nabla_{D'} = \lead \coeff \nabla_{D}.$$
	\end{lemma}

	\begin{proof}
		
		Let $D$ be a positive link diagram whose reduced $A$-state graph $G$ has $t$ cut edges (edges that are not part of cycles).
		By Lemma \ref{smooth away, keep lead coeff}, it suffices to consider the case where every edge in $G$ corresponds to exactly one crossing in $D$. 
		
		Then there are $t$ nugatory crossings in $D$. For each such crossing, we can lift the overstrand up and over half of the diagram
		without changing any other part of the diagram, and the only effect this has on the reduced $A$-state graph is to contract the corresponding edge. This new diagram then represents the same link as the original and so has the same Conway polynomial, not just the same leading Conway coefficient. The result follows by Lemma \ref{smooth away, keep lead coeff}. 
	\end{proof}

	\begin{remark}
		This means that to find the leading Conway coefficient of a positive link $D$, we can:
		\begin{itemize}
			\item Draw its reduced $A$-state graph $G$, 
			\item Contract all cut edges to obtain a new graph $G'$,
			\item Draw a positive link diagram $D'$ whose reduced $A$-state graph is $G'$ and whose crossings are in one-to-one correspondence with the edges of $G'$, and then 
			\item Find the leading Conway coefficient for $D'$. 
		\end{itemize}
	\end{remark}

	\subsection{Leading Conway Coefficient for Balanced diagrams of type $1$} \label{lead conway coeff coeff 1 section}

	We begin with the simplest case. For any even integer $k$, the standard alternating diagram of the $T(2,k)$ torus link oriented so that it is positive and bounds an annulus, is a balanced diagram of type $1$. It is known that 
	\begin{proposition}\label{torus links conway polynomial}
		The Conway polynomial of the positive torus link $T(2,k)$ for even $k$ is $$\nabla_{T(2,k)} = \frac{k}{2}z.$$ 
	\end{proposition}
	\begin{proof}
		This can be seen using the Conway skein relation
		
		$$\nabla_{D+} - \nabla_{D_-} = z \nabla_{D_0}$$
		
		and the observation that smoothing one of the crossings gives the unknot, while changing that crossing gives an almost-positive diagram of the positive link $T(2,k-2)$. So 
		$$\nabla_{T(2,k)} = \nabla_{T(2,k-2)} + z = \dots = \underset{=z}{\underbrace{\nabla_{T(2,2)}}} + \Big(\frac{k}{2} - 1 \Big) z = \frac{k}{2}z.$$
	\end{proof}

	\begin{lemma}\label{lead coeff k-Bur}
		Let $D$ be a type $1$ diagram. Let $k$ be the length of the cycle in its reduced $A$-state graph. Then the leading coefficient of its Conway polynomial is $$\lead \coeff \nabla_D = \frac{k}{2}.$$
	\end{lemma}
	
	\begin{proof} 
		Let $D$ as above. If we contract all cut edges in its reduced $A$-state graph, we have a cycle graph with $k$ edges. This is the reduced $A$-state graph of the standard positive diagram of torus link $T(2,k)$. By Lemma \ref{contract cut edges, same result}, $\lead \coeff \nabla_D = \lead \coeff \nabla_{T(2,k)}$, which by Lemma \ref{torus links conway polynomial} is equal to $\frac{k}{2}$. 
	\end{proof}
	
	\printbibliography
	
\end{document}